\documentclass[11pt,a4paper]{article}

\usepackage{color}
\usepackage{graphics,epsfig}
\usepackage{amsfonts,amssymb,amsmath,amsthm}

\newtheorem{lem}{Lemma}[section]
\newtheorem{cor}[lem]{Corollary}
\newtheorem{thm}[lem]{Theorem}
\newtheorem{assumption}[lem]{Assumption}

\newtheorem{defi}[lem]{Definition}
\theoremstyle{remark}
\newtheorem{rem}[lem]{Remark}

\numberwithin{equation}{section}

\newcommand{\ep}{\varepsilon}
\newcommand{\ue}{u^\ep}
\newcommand{\ve}{v^\ep}
\newcommand{\n}{\nabla }
\newcommand{\om}{\Omega}
\newcommand{\ombar}{\overline{\Omega}}
\newcommand{\eun}{\displaystyle{\frac{1}{\ep}}}

\newcommand{\R}{\mathbb{R}}
\newcommand{\vsp}{\vspace{8pt}}
\newcommand{\di}{\displaystyle}
\newcommand{\Pe}{(P^{\;\!\ep})}
\newcommand{\Pz}{(P^{\;\!0})}
\newcommand{\Qe}{(Q^{\;\!\ep})}
\newcommand {\Q}{Q_T}
\newcommand{\support}{\om _t ^{supp}}
\newcommand{\supportzero}{\om _{t_0} ^{supp}}
\newcommand{\gammasupport}{\Gamma _t ^{supp}}
\newcommand{\gammasupportzero}{\Gamma _{t_0} ^{supp}}
\newcommand{\nusupport}{\nu _t ^{supp}}
\newcommand{\CG}{C^\star}
\newcommand{\CGBIS}{C_\star}
\newcommand{\f}{\tilde f _\delta}
\newcommand{\hm}{\alpha _- (\de)}
\newcommand{\hp}{\alpha _+ (\de)}
\newcommand{\h}{a(\de)}
\newcommand{\mm}{\mu(\de)}
\newcommand{\emutt}{e^{\mu (\pm \ep M) t/\ep }}
\newcommand{\emuttmoins}{e^{\mu (- \ep M ) t/\ep }}
\newcommand{\emuttplus}{e^{\mu (\ep M) t/\ep }}
\newcommand{\emuttplusep}{e^{\mu (\ep M) t^\ep/\ep }}
\newcommand{\de}{\delta}
\newcommand{\drift}{(P^{0}_{\ep, \text{drift}})}
\newcommand{\intdrift}{\Gamma ^{\ep, \text{drift}} _0}
\newcommand{\Gammadrift}{\Gamma ^{\ep, \text{drift}}}
\newcommand{\Omegadrift}{\Omega ^{\ep, \text{drift}}}

\title{Convergence to a propagating front in a degenerate Fisher-KPP equation
with advection}
\author{ }
\date{}

\begin{document}

\maketitle \vspace{-20 mm}

\begin{center}

{\large\bf Matthieu Alfaro\footnote{ The first author is supported
by the French �Agence Nationale de la Recherche� within the
project IDEE
(ANR-2010-0112-01).} }\\[1ex]
I3M, UMR CNRS 5149,\\
 Universit\'e de Montpellier 2, CC051,\\
 Place Eug\`ene Bataillon, 34095 Montpellier Cedex 5, France,\\[2ex]

{\large\bf Elisabeth Logak }\\[1ex]
D\'epartement de Math\'ematiques, UMR CNRS 8088,\\
 Universit\'e de Cergy-Pontoise, \\
2 avenue Adolphe Chauvin, 95302 Cergy-Pontoise Cedex, France. \\[2ex]

\end{center}

\vspace{15pt}

%\tableofcontents

\begin{abstract}

We consider a Fisher-KPP equation with density-dependent diffusion
and advection, arising from a chemotaxis-growth model. We study
its behavior as a small parameter, related to the thickness of a
diffuse interface, tends to zero. We analyze, for small times, the
emergence of transition layers induced by a balance between
reaction and drift effects. Then we investigate the propagation of
the layers. Convergence to a free-boundary limit problem is proved
and a sharp estimate of the
thickness of the layers is provided.\\


\end{abstract}

%\tableofcontents

\section{Introduction}\label{intro-poreux}

In this paper we consider a Fisher-KPP equation with
density-dependent diffusion and advection, namely
\[
 \Pe \quad\begin{cases}
 u_t=\ep \Delta (u^m)-\n \cdot(u\n \ve)+\eun u(1-u) &\text{in }(0,\infty) \times \om\vspace{3pt}\\
 \di \frac{\partial (u^m)}{\partial \nu} = 0 &\text{on }(0,\infty)\times \partial \om  \vspace{3pt}\\
 u(0,x)=u_0(x) &\text{in }\om\,,
 \end{cases}
\]
with $\ep >0$ a small parameter and $\ve(t,x)$ a smooth given
function. Here $\om$ is a smooth bounded domain in $\R^N$ ($N\geq
2$), $\nu$ is the Euclidian unit normal vector exterior to
$\partial \om$ and $m\geq 2$. We are concerned with the behavior
of the solutions $\ue(t,x)$ as $\ep \to 0$.

\begin{assumption}[Initial data]\label{H2} Throughout this paper, we make the following assumptions on the
initial data.
\begin{itemize}
\item [(i)]Let $\Omega_0$ be a nonempty open bounded set with a
smooth boundary and such that  $\overline{\Omega_0}\subset
\Omega$. Let $\widetilde {u_0}:\overline{\Omega_0}\to \R$ be
$C^0$ in $\overline{\Omega_0}$ and $C^2$ in $\Omega_0$, strictly positive on
$\Omega_0$ and such that $\widetilde {u_0}(x)=0$ for all
$x\in\partial \Omega _0$. Define the map $u_0:\Omega \to \R$ by
\begin{equation*}
u_0(x):=\begin{cases} \widetilde {u_0}(x)&\text{ if } x\in\overline{\Omega_0}\\
0&\text{ if } x\in \Omega \setminus
\overline{\Omega_0}\,.\end{cases}
\end{equation*}
\item [(ii)] $\Omega _0$ is convex. \item [(iii)] there exists
$\delta
>0$ such that, if $n$ denotes the Euclidian unit normal vector
exterior to the ``initial interface" $\Gamma _0:=\partial
\Omega_0$, then
\begin{equation}\label{pente}
\left| \frac{\partial \widetilde{u_0}}{\partial n}(y)\right|\geq
\delta \quad\text{ for all } y\in \Gamma _0\,.
\end{equation}
\end{itemize}
\end{assumption}

\begin{rem}
Note that the comparison principle allows to relax the regularity
assumption on $\widetilde {u_0}$. See \cite{A-Duc} for details.
%More precisely, for $\widetilde{u_0}:\Omega _0 \to \R$ only
%bounded and positive on $\Omega _0$, it is enough to assume the
%existence of a map $u_0 ^-:\overline{\Omega_0}\to \R$ of the class
%$C^2$, positive on $\Omega _0$ and such that $u_0 ^-(x)=0$ for all
%$x\in\partial \Omega$, such that
%\begin{equation*}
%u_0 ^-\leq \widetilde {u_0}\quad\text{ on $\Omega_0$ }\quad\text{
%and }\quad \left| \frac{\partial u_0 ^-}{\partial n}(y)\right|\geq
%\delta \quad\text{ for all } y\in \Gamma _0=\partial\Omega_0\,,
%\end{equation*}
%for some $\delta >0$.
\end{rem}

\begin{assumption}[Structure of $\ve$]\label{H3} We assume that
\begin{equation}\label{vep}
\ve(t,x)=v(t,x)+\ep v^\ep _1 (t,x)\,,
\end{equation}
with $v$ and $v^\ep _1$ smooth functions on $[0,\infty) \times
\overline \Omega$. We assume that, for all $T>0$, there exists $C
>0$ such that, for all $\ep>0$ small enough, it holds that $\Vert
v^\ep _1 \Vert _{C^{1,2}([0,T]\times\overline \Omega)} \leq C$.
Finally we assume
\begin{equation}\label{neumann}
\di \frac{\partial v^\ep}{\partial \nu} = 0 \quad \text{ on
}(0,\infty)\times \partial \om\,.
\end{equation}
\end{assumption}

\begin{rem}\label{rem:extend} In the sequel we smoothly extend $v(t,x)$ in
time and space on the whole of $\R \times \R ^N$, as well as $v^\ep _1(t,x)$ in space
on $[0,\infty)\times \R^N$.  Moreover since we are investigating local in time phenomena we will
assume in the sequel, without loss of generality, that the extensions $v(t,x)$
and $v^\ep _1(t,x)$ vanish outside of a large time-space ball.
\end{rem}

Problem $\Pe$ is a simpler version of a chemotaxis-growth system
with logistic nonlinearity, where $\ve (t,x)$ is not a given
function but is coupled to $u$ either through the parabolic
equation
 $\ep v_t=\Delta v +u-\gamma v$ or through the
 elliptic equation $0=\Delta
v +u-\gamma v$,  supplemented
with the Neumann boundary condition \eqref{neumann} (see e.g.
\cite{Mim-Tsu}).  Note that, in the
case of linear diffusion (corresponding to $m=1$) and a bistable nonlinearity,
the asymptotic behavior of the
 corresponding system as $\ep \rightarrow 0$ has been studied using the Green's function
associated to the homogeneous Neumann boundary value problem on
$\om$ for the operator $-\Delta+ \gamma$ (see
\cite{Bon-Hil-Log-Mim1} and \cite{A-chemo}) .

\vskip 4pt \noindent {\bf Motivation and biological background.}
Before describing our results, let us briefly comment about  the
relevance of $\Pe$ in population dynamics models. The evolution
equation in Problem $\Pe$ combines logistic growth, chemotaxis and
degenerate diffusion. We recall below how these terms appear in
mathematical models
 that attempt to capture remarkable biological features.

Reaction diffusion equations with a logistic nonlinearity were
introduced in the pioneering works \cite{Fish}, \cite{Kol-Pet-Pis}.
The simplest equation reads
 $$u_t=\Delta u +u(1-u)\,,$$
and has been widely used to model phenomena arising in population
genetics \cite{Fish} or in biological invasions \cite{Shi-Kaw}.
Its main mathematical property is to sustain travelling wave
solutions with a semi-infinite interval of admissible wave speeds,
with  the minimal one having a crucial  biological interpretation.

Chemotaxis, i.e. the tendency of biological individuals to direct
their movements according to certain chemicals in their
environment, is induced in $\Pe$ by the advection term $-\n
\cdot(u\n \ve)$: the population, whose density is $u(t,x)$, has an
oriented motion in the direction of a positive gradient of the
chemotactic substance, whose concentration is $\ve(t,x)$. The
first PDE model to describe such movements was proposed in
\cite{Kel-Seg} and involves linear diffusion for $u$ and a
parabolic equation coupling $v$ to $u$. The Keller-Segel model has
received considerable attention in mathematical literature,
particularly focusing on the finite-time blow-up of solutions (see
\cite{hillen} for a recent review). This provides a mathematical
tool to analyze  aggregation phenomena as observed in bacteria
colonies. Chemotaxis systems involving linear diffusion and a
growth term, either logistic or bistable, have later been
considered in, e.g., \cite{Mim-Tsu}, \cite{Bon-Hil-Log-Mim1},
\cite{A-chemo} and \cite{winkler} .

Variants of the  Fisher-KPP equation involving a degenerate
diffusion have been proposed in order to take into account
 population density pressure. Actually one can introduce density-dependent birth or death rates as an attempt to
 control the size of  a population. Nevertheless as shown in \cite{GN}, the introduction of a
nonlinearity into the dispersal behavior of a species, which
behaves in an otherwise linear way, may lead, in an inhomogeneous
environment, to a similar regulatory effect. Moreover this assumption is consistent with ecological observations as
reported for instance in \cite{Carl}, where it is shown that arctic ground squirrels
migrate from densely populated areas into sparsely populated
areas, even when the latter is less favorable (due to reduced availability of burrow sites
 or exposure to intensive predation). For such species, migration to
avoid crowding, rather than random motion, is the primary cause of
dispersal. To describe such movements, the authors in  \cite{Shi-Kaw}  and \cite{GN}
use the directed motion model where individuals can only stay put
or move down the population gradient; this model yields the
degenerate equation
\begin{equation}\label{gurney}
u_t=\Delta (u^2)+G(x)u\,,
\end{equation}
in which the  population regulates its size below the carrying
capacity set by the supply of nutrients.
Later in \cite{GM} a larger class of equations with degenerate
diffusion and nonlinear reaction was considered, namely
\begin{equation}\label{general}
u_t=\Delta (u^m)+f(u)\,,\quad \quad m\geq 2\,.
\end{equation}
Note that in the absence of $f(u)$, equation \eqref{general} reduces to the
so-called porous medium equation
\begin{equation}\label{pme}
u_t=\Delta (u^m)\,,
\end{equation}
which has been extensively investigated in the literature. We
refer to the book \cite{V} and the references therein. The main
feature of this equation is that it is degenerate at the points
where $u=0$. As a consequence, a loss of regularity of solutions
occurs and disturbances propagate with finite speed, a property
which has a relevant interpretation in a biological context (see
for instance \cite{reb}).

\vskip 4pt \noindent {\bf Formal asymptotic analysis.} Problem
$\Pe$ possesses a unique solution $\ue(t,x)$ in a sense that is
explained in  Section \ref{s:existence}. As $\ep \rightarrow 0$,
the qualitative behavior of this solution is the following. In the
very early stage, the nonlinear diffusion term $\ep \Delta (u^m)$
is negligible compared with the drift term $-\n u\cdot \n \ve$ and
the reaction term $\ep^{-1} u(1-u)$. Hence, in some sense, the
equation is well approximated by a coupling between the transport
equation $u_t+\n u\cdot \n \ve=0$ and the ordinary differential
equation $u_t=\ep ^{-1} u(1-u)$. Therefore, as suggested by the
analysis in \cite{A-Duc}, $\ue$ quickly approaches the values $0$
or $1$, and an interface is formed between the regions
$\{\ue\approx 0\}$ and $\{\ue\approx 1\}$ ({\it emergence of the
layers}). Note that, in this very early stage, the balance of the
transport equation and the ordinary differential equation will
generate an interface not exactly around $\Gamma _ 0$ but in a
slightly drifted place. Once such an interface is developed, the
diffusion term becomes large near the interface, and comes to
balance with the drift and the reaction terms so that the
interface starts to propagate, on a much slower time scale ({\it
propagation of the front}).

\vsp Our goal in this paper is to provide a rigorous analysis that supports this
formal approach and makes it more precise. To study the interfacial behavior, we consider
the asymptotic limit of $\Pe$ as $\ep\rightarrow 0$. Then the
limit solution will be a step function $\tilde u (t,x)$ taking the
value $1$ on one side of a moving interface, and $0$ on the other
side. We show that this sharp interface, which we will denote by $\Gamma_t$,
obeys the law of motion
\[
 \Pz\quad\begin{cases}
 \, V_{n}=c^*+\di \frac{\partial v}{\partial n}
 \quad \text { on } \Gamma_t \vspace{3pt}\\
 \, \Gamma_t\big|_{t=0}= \Gamma_0\,,
\end{cases}
\]
where $V_n$ is the normal velocity of $\Gamma _t$ in the exterior
direction, $c^*$ the minimal speed of travelling waves solutions
of a related degenerate one-dimensional problem (see Section
\ref{s:motion} for details) and $n$ the outward normal vector on
$\Gamma _t$.

\vskip 4pt\noindent {\bf Plan.} The organization of this paper is
as follows. We present our results in Section \ref{s:results}. In
Section \ref{s:existence}, we briefly recall known results
concerning the well-posedness of Problem $\Pe$; in particular, it
admits a comparison principle so that the sub- and super-solutions
method can be used to investigate the behavior of the solutions
$\ue$. In Section \ref{s:generation}, we prove a generation of
interface property for Problem $\Pe$. In Section \ref{s:motion} we
investigate the motion of interface. Finally, we prove our main
result in Section \ref{s:proof}.

\section{Results and comments}\label{s:results}

The question of the convergence of Problem $\Pe$ to $\Pz$ has been
addressed in \cite{Dkh}. However, the author considers only a very
restricted class of initial data, namely those having a specific
profile with well-developed transition layers. In other words the
generation of interface from arbitrary initial data is not
studied. In the present paper we study both the emergence and the
propagation of interface. Moreover we prove a sharp $\mathcal
O(\ep)$ estimate of the thickness of the transition layers of the
solutions $\ue$.

The authors in \cite{HKLM} prove the convergence of the solutions
of $\Pe$  with arbitrary initial data with convex compact support
to solutions of $\Pz$, when there is no advection (i.e. $\ve
\equiv 0$). They provide an $\mathcal O (\ep|\ln \ep|)$ estimate
of the thickness of the transition layers. Therefore, even in the
particular case $\ve \equiv 0$, our $\mathcal O (\ep)$ estimate
was not known.

\vskip 4pt %Let us recall that the sharp interface limit of $u_t=\Delta
%u-\n\cdot(u\n \ve)+\edeux f(u)$ with $f$ of the bistable type is
%studied in \cite{A-chemo}: in the very early stage the reaction
%term is dominant and the generation of interface is initiated by
%the dynamics of the ODE $u_t=\edeux f(u)$.
As mentioned in the introduction, the drift term and the reaction
term in $\Pe$ are of the same magnitude for small times. Therefore
the emergence of the layers, initiated by the ODE $u_t=\ep ^{-1}
u(1-u)$, will occur in the neighborhood of a slightly drifted
initial interface $\intdrift$. To analyze such a phenomenon we
shall use the Lagrangian coordinates. Recall that we have smoothly
extended $v(t,x)$ in time-space on the whole of $\R \times \R ^N
$, with $v\equiv 0$ outside of a large time-space ball. Then, for
$(t_0,x_0) \in \R \times \R ^N$, we denote by
$\varphi_{(t_0,x_0)}$ the solution, defined on $\R$, of the Cauchy
problem
\begin{equation}\label{lagrange}
\begin{cases}
\di \frac {d X}{d t}(t)=\n v \left(t,X(t)\right)\,,\vsp\\
X(t_0)=x_0\,.
\end{cases}
\end{equation}
We denote by $\Phi$ the associated flow defined on $\R \times \R
\times \R ^N$, that is
\begin{equation}\label{flow}
\Phi (t_1,t_2,x_3):=\varphi_{(t_2,x_3)}(t_1)\,.
\end{equation}
Recall that $\Gamma _0 =\partial \Omega _0=\partial ({\rm Supp}\;
u_0)$ is the initial interface. From $t=0$ to
\begin{equation}\label{time}
t^\ep:=\ep|\ln \ep|\quad\ \hbox{(generation time)}\,,
\end{equation}
we let each point on $\Gamma _0$ evolve with the law
\eqref{lagrange} and then define a {\it drifted initial interface}
$\intdrift$ by
\begin{equation}\label{drifted-interface}
\intdrift:=\{\Phi(t^\ep,0,x):\, x\in \Gamma _0\}\,.
\end{equation}
Next we consider the free boundary problem
\[
 \drift \quad\begin{cases}
 \, V_{n}=c^*+\di \frac{\partial v}{\partial n}
 \quad \text { on } \Gammadrift _t \vspace{3pt}\\
 \, \Gammadrift _t\big|_{t=0}= \intdrift.
\end{cases}
\]

\vskip 4pt \noindent{\bf Well-posedness of $\Pz$ and of $
\drift$.} Using the level set formulation (see, e.g., \cite{bss}),
the motion law in Problem $\Pz$ can be rewritten as a first order
Hamilton-Jacobi equation with a convex Hamiltonian. This approach,
combined with the results in \cite{ley}, has been used in
\cite{Dkh} in order to prove the following.

\begin{thm}[\cite{Dkh}, Well-posedness of $\Pz$] Let $\Omega_0 \subset \subset \Omega$ be a smooth subdomain
of $\Omega$ and let
 $\Gamma _0= \partial \Omega_0$
be the given smooth initial interface. Then there exists
$T^{max}(\Gamma _0)>0$ such that Problem $\Pz$ has a unique smooth
solution on $[0,T]$ for any $0<T<T^{max}(\Gamma_0)$. More
precisely, there exists a family of smooth subdomains $(\Omega_
t)_{t \in (0,T]}$ with
 $\Omega_ t \subset \subset \Omega$  such that, denoting $\Gamma_t  =\partial \Omega_ t$,
 $\Gamma
:=\bigcup _{0\leq t \leq T} (\{t\}\times \Gamma_t)$ is the unique
solution to Problem $\Pz$ on $[0,T]$.
\end{thm}

Moreover, $T^{max}(\Gamma _0)$ depends smoothly on $\Gamma _0$.
 Therefore we can choose  $\ep_0 >0$
small enough and $T>0$ such that
\begin{equation}\label{choixT}
0< T < \inf_{0\leq \ep \leq \ep_0} T^{max}(\intdrift)\,,
\end{equation}
which guarantees the existence of a unique smooth solution on
$[0,T]$ to both $\Pz$ and $ \drift$ for any $0<\ep\leq\ep_0$. We
denote by $\Gammadrift =\bigcup _{0\leq t \leq T} (\{t\}\times
\Gammadrift _t)$ the smooth solution to $\drift$ and by
$\Omegadrift _t$ the region enclosed by $\Gammadrift _t$. In the
sequel we work on $[0,T]$, with $T$ satisfying (\ref{choixT}), and
define $\Q:=(0,T)\times \Omega$.

\vskip 4pt Our main result, Theorem \ref{th:results}, contains
generation, motion and thickness of the transition layers
properties. It asserts that: given an initial data $u_0$, the
solution $\ue$ quickly (at time $t^\ep=\ep|\ln \ep|$) becomes
close to 1 or 0, except in a small neighborhood of the drifted
interface $\Gammadrift _{t^\ep}$, creating a steep transition
layer around $\Gammadrift _{t^\ep}$ ({\it generation of
interface}). The theorem then states that the solution $\ue$
remains close to the step function associated with $\drift$ on the
time interval $[t^\ep,T]$ ({\it motion of interface}); in other
words, the motion of the transition layer is well approximated by
the limit interface equation $\drift$. Moreover, the estimate
\eqref{resultat} in Theorem \ref{th:results} implies that, once a
transition layer is formed, its thickness remains within order
$\mathcal O(\ep)$ for the rest of time.

\begin{thm}[Generation, motion and thickness of the layers]\label{th:results}
Let $\eta \in (0,1/2)$ be arbitrary. Then, there exists $\mathcal
C
>0$ such that, for all $\ep
>0$ small enough and all
$$
t^\ep=\ep  |\ln \ep|\leq t \leq T\,,
$$
we have
\begin{equation}\label{resultat} \ue(t,x) \in
\begin{cases}
\,[0,1+\eta]\quad&\text{if}\quad
x\in\mathcal N_{\mathcal C\ep}(\Gammadrift _t)\\
\,[1-\eta,1+\eta]\quad&\text{if}\quad x\in\Omegadrift  _
t\setminus\mathcal N_{\mathcal C\ep}(\Gammadrift
_t)\\
\,\{0\}\quad&\text{if}\quad x\in (\Omega \setminus
\overline{\Omegadrift _t})\setminus\mathcal N_{\mathcal
C\ep}(\Gammadrift _t)\,,
\end{cases}
\end{equation}
with $\mathcal N _r(\Gammadrift  _t):=\{x:\,{\rm
dist}(x,\Gammadrift
 _t)<r\}$ the tubular $r$-neighborhood of $\Gammadrift
 _t$.
\end{thm}

Note that \eqref{resultat} shows that, for any $0<a<1$, for all
$t^\ep\leq t \leq T$, the $a$-level set $$ L _t ^\ep (a):=\{x:\,
\ue(t,x)=a\}$$ lives in a tubular $\mathcal O (\ep)$ neighborhood
of the  interface $\Gammadrift _t$. In other words, we provide a
new $\mathcal O(\ep)$ estimate of the thickness of the transition
layers of the solutions $\ue$. Concerning the localization of the
level sets $L _t ^\ep (a)$, it is made with respect to a slightly
drifted free boundary problem $\drift$. Nevertheless, since the
solution of $\drift$ on $[0,T]$ is continuous w.r.t. the initial
hypersurface $\intdrift$, we recover, as $\ep \to 0$, the original
free boundary problem $\Pz$ and obtain the expected result. More
precisely, let us define the step function $\tilde u (t,x)$ by
\begin{equation}\label{u}
\tilde u  (t,x):=\begin{cases}
\, 1 &\text{ in } \Omega  _t\\
\, 0 &\text{ in } \Omega \setminus \overline{\Omega  _t}
\end{cases} \quad\text{for } t\in(0,T]\,.
\end{equation}
 As a consequence of Theorem \ref{th:results},
we obtain the following convergence result which shows that $\tilde u$ is the
sharp interface limit of $\ue$ as
$\ep\to 0$.
\begin{cor}[Convergence]\label{cor:cv} As $\ep\to 0$, $\ue$ converges to
$\tilde u $, defined in \eqref{u}, everywhere in $\bigcup _{0<
t\leq T}(\{t\}\times \Omega _t)$ and $\bigcup _{0< t \leq T}\left(
\{ t\}\times(\Omega \setminus \overline{\Omega  _t})\right)$.
\end{cor}

%\begin{rem}[When the drift is small]Assume that in \eqref{vep} we
%have $v(t,x)=v(t)$ so that $v^\ep (t,x)= v(t)+\mathcal O (\ep)$.
%Then, for small times, the drift does not occur and in the above
%results $\intdrift$, $\drift$, $\Gammadrift _t$, $\Omegadrift _t$
%etc are replaced by $\Gamma _ 0$, $\Pz$, $\Gamma _t$, $\Omega _t$
%etc. Therefore we can localize the $a$-level sets of $\ue$ with
%respect to the \lq\lq correct" limit free boundary problem $\Pz$.
%Note again that our $\mathcal O (\ep)$ estimate of the thickness
%of the transition layers improves the $\mathcal O (\ep |\ln \ep|)$
%proved in \cite{HKLM} when $\ve \equiv 0$.
%\end{rem}

\section{Comparison principle, well-posedness for $\Pe$}\label{s:existence}

Since the diffusion term degenerates when $u=0$ a loss of
regularity of solutions occurs. We define below a notion of weak
solution for Problem $\Pe$, which is very similar to the one
proposed in \cite{ACP} for the one dimensional problem with
homogeneous Dirichlet boundary conditions. Concerning the initial
data, we suppose here that $u_0 \in L^\infty (\om)$ and $u_0 \geq
0$ a.e. Note that in this section, and only in this section, we
assume, for ease of notation, that $\ep=1$ and that $v^\ep \equiv
v$; we then denote the associated Problem $\Pe$ by $(P)$. In the
sequel $f(u)=u(1-u)$.
%\[
% (P) \quad\begin{cases}
% u_t=\Delta (u^m)+f(u) &\text{in }\om \times (0,\infty)\vspace{3pt}\\
% \di \frac{\partial (u^m)}{\partial \nu} = 0 &\text{on }\partial \om \times (0,\infty)\vspace{3pt}\\
% u(0,x)=u_0(x) &\text{in }\om\,,
% \end{cases}
%\]
%with $f(u):=u(1-u)$.

\begin{defi}\label{definition-weaksol-poreux}
A function $u:[0,\infty)\to L^1(\om)$ is a solution of Problem
$(P)$ if, for all $T>0$,
\begin{enumerate}
\item [(i)]$u \in C\left([0,\infty);L^1(\om)\right)\cap L^\infty
(\Q)$
  \item [(ii)] for all $\varphi \in C^2(\overline {\Q})$ such that $\varphi \geq
  0$ and $\di \frac {\partial \varphi}{\partial \nu}=0$ on $\partial
  \om$, it holds that
\begin{eqnarray}
\int_ \om u(T)\varphi(T)-\int\int_{\Q}(u\varphi _t+u^m\Delta
\varphi+u\n v \cdot \n \varphi)\nonumber\\
=\int _\om u_0\varphi(0)+\int\int _{\Q}
f(u)\varphi\,.\label{deqexi-poreux}
\end{eqnarray}
\end{enumerate}
A sub-solution (a super-solution) of Problem $(P)$ is a function
satisfying $(i)$ and $(ii)$ with equality replaced by $\leq$
(respectively $\geq$).
\end{defi}

\begin{thm}[Existence and
comparison principle]\label{Existence-comparison}Let $T>0$ be
arbitrary. The following properties hold.
\begin{enumerate}
 \item [(i)] Let $u^-$ ($u^+$) be a sub-solution
(respectively a super-solution) with initial data $u_0^-$
(respectively $u_0^+$).
$$
\text{ If }\quad u_0^- \leq u_0^+ \,\text{ a.e.} \quad\text{ then
}\quad u^-\leq u^+ \,\text{ in $Q_T$;}
$$
 \item [(ii)] Problem $(P)$ has a unique solution $u$ on
$[0,\infty)$ and
\begin{equation}\label{encadrement}
0\leq u \leq \max(1,\Vert u_0 \Vert _{L^\infty(\om)})\, \text{ in
$Q_T$;}
\end{equation}
\item [(iii)] $u \in C(\overline{Q_T})$.
\end{enumerate}
\end{thm}

Since \eqref{neumann} holds, the proof of Theorem
\ref{Existence-comparison} is standard and follows the same steps
of that of \cite[Theorem 5]{ACP}. The continuity of $u$ follows
from \cite{DB}.

\vskip 4pt The following lemma proved in \cite{HKLM}, will be very
useful when constructing smooth sub- and super-solutions in later
sections.

\begin{lem}[Being sub- and super-solutions]\label{lemma-sup-poreux}
Let $u$ be a continuous nonnegative function in $\overline {\Q}$.
Define $$\support:=\{x\in\om:\, u(t,x)>0\}\qquad
\gammasupport:=\partial \support\,,$$
 for all $t\in[0,T]$. Suppose
the family $\Gamma ^{supp}:=\bigcup _{0\leq t \leq T}  \left(\{t\}
\times \gammasupport\right)$ is sufficiently smooth and let
$\nusupport$ be the outward normal vector on $\gammasupport$.
Suppose moreover that
\begin{enumerate}
\item [(i)] $\n (u^m) \; \text{ is continuous in }\ \overline
{\Q}$ \item [(ii)] ${\mathcal L}^\ep[u]:=u_t-\Delta
(u^m)+\n\cdot(u\n v)-f(u)=0\; \text{ in }\ \{(t,x)\in \overline
{\Q}:\, u(t,x)>0\}$ \item [(iii)] $\di{\frac{\partial
(u^m)}{\partial \nusupport}}=0\; \text{ on }\
\partial \support,\; \text{ for all }\ t\in[0,T]$\,.
\end{enumerate}
Then $u$ is a solution of Problem $(P)$. Similarly  a function
satisfying $(i)$ and $(ii)$---$(iii)$ with equality replaced by
$\leq$ ($\geq$) is a sub-solution (respectively a super-solution)
of Problem $(P)$.
\end{lem}

\section{Emergence of the transition layers}\label{s:generation}

In this section, we investigate the generation of interface which
occurs very quickly around $\Gammadrift _{t^\ep}$. We prove that,
given a virtually arbitrary initial datum $u_0$, the solution
$\ue$ of $\Pe$ quickly becomes close to $1$ or $0$ in most part of
$\Omega$. More precisely --- recalling that $\Phi(t_1,t_2,x_3)$,
defined in \eqref{flow}, denotes the flow associated with the
Cauchy problem \eqref{lagrange}--- the following holds.

\begin{thm}[Emergence of the layers]\label{th-gen-poreux}Let
$\eta\in(0,1/2)$ be arbitrary. Then there exists $M_0>0$ such
that, for all $\ep>0$ small enough, the following holds with
$t^\ep=\ep|\ln \ep|$.

\begin{enumerate}
\item [(i)] for all $x\in\om$, we have that
\begin{equation}\label{facile}
0 \leq u^\ep(t^\ep, x) \leq 1+\eta\,;
\end{equation}
\item [(ii)] for all $x\in\om$, we have that
\begin{equation}\label{sub-fisher}
\text{ if } \quad u_0(\Phi(0,t^\ep,x))\geq M_0\ep\quad \text{ then
}\quad u^\ep(t^\ep, x) \geq 1-\eta\,;
\end{equation}
\item [(iii)] for all $x\in\om$, we have
\begin{equation}\label{super-fish}
\text{ if }\quad {\rm dist}(\Phi(0,t^\ep,x),\Omega _0)\geq M_0 \ep
\quad \text{ then }\quad u^\ep(t^\ep, x)=0\,,
\end{equation}
where we recall that $\Omega _0=\{x:\,u_0(x)>0\}$ (see Assumption
\ref{H2}).
\end{enumerate}
\end{thm}

In order to prove the above theorem, we shall construct sub- and
super-solutions. As mentioned before, in this very early stage, we
have to take into account both the reaction and the drift terms.
We start with some preparations.

\subsection{A related ODE and the flow $\Phi$}\label{ss:ode}

\noindent{\bf An ODE.} The solution of the problem without
diffusion nor advection, namely $\bar{u}_t=\ep ^{-1}\,f(\bar{u})$
supplemented with the condition $\bar{u}(0,x)=u_0(x)$,  is written
in the form $\bar{u}(t,x)=Y\left(\frac{t}{\ep},\,u_0(x)\right)$,
where $Y(\tau,\xi)$ denotes the solution of the ordinary
differential equation $Y_\tau (\tau,\xi)=f(Y(\tau,\xi))$
supplemented with the initial condition  $Y(0,\xi)=\xi$.
Nevertheless, in order to take care of the term $-u\Delta \ve$, we
need a slight modification of $f$.

Let $\tilde f$ be the smooth odd function that coincides with
$f(u)=u(1-u)$ on $[0,\infty)$: $\tilde f$ has exactly three zeros
$-1<0<1$ and
\begin{equation}
{\tilde f}'(-1)=-1<0\,, \qquad \tilde f'(0)=1>0\,, \qquad \tilde
f'(1)=-1<0\,,
\end{equation}
i.e. $\tilde f$ is of the bistable type. Next, we define
$$
\f(u):=\tilde f(u)+\delta\,.
$$
For $|\delta|$ small enough, this function is still of the
bistable type: if $\de _0$ is small enough, then for any
$\de\in(-\de_0,\de_0)$, $\f$ has exactly three zeros $\hm < \h <
\hp$ and there exists a positive constant $C$ such that
\begin{equation}\label{h}
|\hm +1|+|\h|+|\hp-1|\leq C|\de|\,,
\end{equation}
\begin{equation}\label{mu}
| \mm -1| \leq C|\de|\,,
\end{equation}
where $\mm$ is the slope of $\f$ at the unstable zero, namely
\begin{equation}\label{slope}
\mm:= {\f} '(\h)={\tilde f}'(\h)\,.
\end{equation}

Now for each $\de\in(-\de_0,\de_0)$, we define $Y(\tau,\xi\,;\de)$
as the solution of the ordinary differential equation
\begin{equation}\label{ode-poreux}
\left\{\begin{array}{ll} Y_\tau (\tau,\xi\,;\de)&=\f
(Y(\tau,\xi\,;\de)) \quad \text { for }\tau >0\vspace{3pt}\\
Y(0,\xi\,;\de)&=\xi\,,
\end{array}\right.
\end{equation}
where $\xi$ varies in $(-C_0,C_0)$, with
\begin{equation}\label{def:cste}
C_0:=\Vert u_0 \Vert _{L^\infty(\Omega)}+1\,. \end{equation}
We
claim that $Y(\tau,\xi\,;\de)$ has the following properties.

\begin{lem}[Behavior of $Y$]\label{properties-Y}There exist positive
constants $\de_0$ and $C$ such that the following holds for all
$(\tau,\xi\,;\de) \in (0,\infty)\times (-C_0,C_0)\times (-\de
_0,\de _0)$.
\begin{enumerate}
\item [(i)]
$\text{ If }\quad \xi >a(\de) \quad\text{ then }\quad Y(\tau,\xi\,;\de)>a(\de)$\\
$\text{ If }\quad \xi <a(\de) \quad\text{ then }\quad
Y(\tau,\xi\,;\de)<a(\de)$ \item [(ii)] $|Y(\tau,\xi\,;\de)|\leq
C_0$ \item [(iii)] $0< Y _ \xi (\tau,\xi\,;\de)\leq C e^{\mm
\tau}$ \item [(iv)]
$\left|\di{\frac{Y_{\xi\xi}}{Y_\xi}(\tau,\xi\,;\de)}\right|\leq C
(e^{\mm \tau}-1)$.
\end{enumerate}
\end{lem}

Properties $(i)$ and $(ii)$ are direct consequences of the
bistable profile of $\f$.  For proofs of $(iii)$ and $(iv)$ we
refer to \cite{A-chemo}.

\vskip 4pt \noindent {\bf The flow $\Phi$.} Let us briefly recall
well known facts concerning the flow $\Phi(t_1,t_2,x_3)$. By
definition we have
\begin{equation}\label{edp-flow}
\frac {\partial \Phi}{\partial t_1}(t,t_0,x_0)=\n
v(t,\Phi(t,t_0,x_0))\,.
\end{equation}
Next, note that, by uniqueness,
$$
\Phi(t,t_0,\Phi(t_0,t,x_0))=x_0\,, $$ for all $(t,t_0,x_0)\in\R
\times \R \times \R ^N$. Differentiating this identity with
respect to $t_0$, we get
$$
\frac {\partial \Phi}{\partial t_2} (t,t_0,x)+D_3\Phi (t,t_0,x)
\frac {\partial \Phi}{\partial t_1}(t_0,t,x_0)=0_{\R^N}\,,
$$
where $x:=\Phi(t_0,t,x_0)$ and where $D_3 \Phi(t_1,t_2,x_3)$
denotes the Jacobian matrix of $\Phi$ w.r.t. the third variable.
Hence, using \eqref{edp-flow} we infer that
\begin{equation}\label{egalite}
\frac {\partial \Phi}{\partial t_2} (t,t_0,x)+D_3\Phi (t,t_0,x) \n
v (t_0,x)=0_{\R^N}\,,
\end{equation}
which is of crucial importance for our analysis.

\subsection{Proof of \eqref{facile} and \eqref{sub-fisher}}\label{ss:gen1}

We use the notation $z^+=\max(z,0)$. Our sub- and super-solutions
are given by
\begin{equation}\label{w+-}
w_\ep^\pm(t,x):=\left[Y\left(\frac{t}{\ep},\,u_0(\Phi(0,t,x))\pm\ep^2\CG(\emutt-1)\,;\pm
\ep M\right)\right]^+\,,
\end{equation}
or equivalently by
\begin{equation}\label{w+-2}
w_\ep^\pm(t,\Phi(t,0,x)):=\left[Y\left(\frac{t}{\ep},\,u_0(x)\pm\ep^2\CG(\emutt-1)\,;\pm
\ep M\right)\right]^+\,.
\end{equation}
Here $Y(\tau,\xi\,;\delta)$ is the solution of \eqref{ode-poreux},
$\mm$ the slope defined in \eqref{slope}, $\Phi(t_1,t_2,x_3)$ the
flow defined in \eqref{flow} and $M$ is chosen such that, for all
$\ep>0$ small enough, $M\geq C_0 \Vert \Delta \ve \Vert
_{L^\infty(\Q)}$, with $C_0$ defined by \eqref{def:cste}.

\begin{lem}[Sub- and super-solutions for small times]\label{g-w}
There exists $\CG >0$ such that, for all $\ep>0$ small enough,
$(w_\ep^-,w_\ep^+)$ is a pair of sub- and super-solutions for
Problem $\Pe$, in the domain $ [0, t^\ep]\times \ombar$.
\end{lem}

Before proving the lemma, we remark that $w^-_\ep(0,x)=w^+ _\ep
(0,x)=u_0(x)$. Consequently, by the comparison principle, we have
\begin{equation}\label{g-coincee1}
w_\ep^-(t,x) \leq u^\ep(t,x) \leq w_\ep^+(t,x) \quad \text{ for
all }\, (t,x)\in [0,t^\ep]\times \ombar\,.
\end{equation}

\begin{proof} In order to prove that $(w_\ep ^-,w_\ep ^+)$ is a pair of
sub- and super-solutions for Problem $\Pe$ --- if $\CG$ is
appropriately chosen--- we check the sufficient conditions stated
in Lemma \ref{lemma-sup-poreux}.

On the one hand, concerning the sub-solution $w_\ep ^-$, for
$(t,x)$ such that $x\in\support [w_\ep^-]:=\{x:\, w_\ep
^-(t,x)>0\}$ we have, at point $(t,x)$,
\begin{eqnarray*}
\n (w_\ep ^-)^m= (mY^{m-1}Y_\xi)\Big(\frac{t}{\ep},\,u_0(\Phi(0,t,x))-\ep^2\CG(\emuttmoins-1)\,;\nonumber\\
-\ep M\Big) \n _x (u_0(\Phi(0,t,x))\,.
\end{eqnarray*}
If $(t,x) \to(t_0,x_0)$ such that $x_0\in \gammasupportzero [w_\ep
^-]:=\partial \supportzero [w_\ep ^-]$ then the equality above
implies
$$
\lim _{(t,x)\to(t_0,x_0)} \n (w_\ep ^-)^m (t,x)=0_{\R ^N}\,,
$$
and conditions $(i)$ and $(iii)$ of Lemma \ref{lemma-sup-poreux}
are checked for the sub-solution.

On the other hand, concerning the super-solution $w_\ep ^+$, note
that $\xi:=u_0(\Phi(0,t,x))+\ep^2\CG(\emuttplus-1)$ is positive.
Therefore the cubic profile of $\f$ shows that, for $t>0$,
$$
\begin{array}{ll}
\support[w_\ep^+]=\om\vspace{3pt}\\
\gammasupport[w_\ep ^+]:=\partial \support [w_\ep ^+]=\partial
\om\,.
\end{array}
$$
Recall that $u_0=0$ in a neighborhood $\mathcal V$ of $\partial
\Omega$; if $x$ is sufficiently close to $\partial \Omega$,
$\Phi(0,t,x)$ lives in $\mathcal V$ for all $t\in[0,t^\ep]$ (with
$\ep>0$ sufficiently small). Therefore \eqref{w+-} shows that
$w_\ep ^+$ is independent on $x$ near $\partial \Omega$ and
condition $(iii)$ of Lemma \ref{lemma-sup-poreux} for the
super-solution is checked (and condition $(i)$ is obviously
checked).

\vsp Then it remains to prove that
$$
{\mathcal L}^\ep [w_\ep ^-]:=(w_\ep ^-)_t-\ep \Delta ({w_\ep
^-})^m+\n\cdot(w_\ep ^-\n \ve)-\eun f(w_\ep ^-)\leq 0\,,
$$
in $\{(t,x)\in  [0,t^\ep]\times \ombar:\, w_\ep ^-(t,x)>0\}$ and
that ${\mathcal L}^\ep [w_\ep ^+] \geq 0$ in $\{(t,x)\in
[0,t^\ep]\times\ombar \}$. We will only prove the latter
inequality since the proof of the former is similar.

We compute
$$
\begin{array}{ll}
\partial _t w_\ep ^+=\eun Y_\tau+\frac{\partial}{\partial
t}\left[u_0(\Phi(0,t,x))\right]Y_\xi+\ep\mu(\ep M)\CG\emuttplus Y_\xi\vsp\\
\n w_\ep ^+=\n _x\left[u_0(\Phi(0,t,x))\right]Y_\xi \vsp\\
\n \left[(w_\ep ^+)^m\right]=\n
_x\left[u_0(\Phi(0,t,x))\right](Y^m)_\xi\vsp\\
\Delta \left [(w_\ep
^+)^m\right]=|\n_x\left[u_0(\Phi(0,t,x))\right]|^2(Y^m)_{\xi\xi}+\Delta
_x \left[u_0(\Phi(0,t,x))\right](Y^m)_\xi\,,
\end{array}
$$
where the function $Y$ and its derivatives are taken at the point
$$(\tau,\xi\,; \de):=(t /{\ep },
u_0(\Phi(0,t,x))+\ep^2\CG(\emuttplus-1)\,;\ep M)\,. $$
 Note that
$$
\frac{\partial}{\partial t}\left[u_0(\Phi(0,t,x))\right]=\n u_0
(\Phi(0,t,x))\cdot \frac{\partial \Phi}{\partial t_2} (0,t,x)\,,
$$
and that
$$
\n_x\left[u_0(\Phi(0,t,x))\right]=(D_3\Phi(0,t,x))^T \n u_0
(\Phi(0,t,x))\,,
$$
with $(D_3 \Phi(t_1,t_2,x_3))^T$ the transpose of the Jacobian
matrix of $\Phi$ w.r.t. the third variable.

Therefore, using $f(w_\ep ^+)=\tilde f(w_\ep ^+)=\tilde f _{\ep
M}(Y)-\ep M$ and the equation $Y_\tau=\tilde f _{\ep M}(Y)$, we
infer that
$$
{\mathcal L} ^\ep[w_\ep ^+]=E_1+E_2+\ep Y _ \xi E_3\,,
$$
where
$$
\begin{array}{ll}
  E_1=& M+Y\Delta \ve\vsp \\
 E_2=& \n u_0 (\Phi(0,t,x))\cdot\left(\frac{\partial \Phi}{\partial t_2}(0,t,x)+D_3\Phi(0,t,x)\n v(t,x)\right)Y_\xi\vsp \\
E_3=&\CG \mu (\ep M)\, \emuttplus+(D_3\Phi(0,t,x))^T \n
u_0(\Phi(0,t,x))\cdot \n v_1^\ep(t,x)\vsp\\&-\Delta
_x\left[u_0(\Phi(0,t,x))\right]
\frac{(Y^m)_\xi}{Y_\xi}-|\n_x\left[u_0(\Phi(0,t,x))\right]|^2\frac{(Y^m)_{\xi\xi}}{Y_\xi}\,.
\end{array}
$$

We note that, for $\ep
>0$ sufficiently small, $\de=\ep M \in(-\de _0,\de_0)$ and that,
in the range $0 \leq t \leq t^\ep= \ep |\ln \ep|$,
$$
\xi=u_0(\Phi(0,t,x))+ \ep ^2 \CG(\emuttplus-1)\in (-C_0,C_0)\,,
$$
so that estimates of Lemma \ref{properties-Y} on $Y(\tau,\xi;\de)$
will apply.

Since we have chosen $M\geq C_0\Vert \Delta \ve \Vert _{L ^\infty
(\Q)}$, $E_1\geq 0$ holds. Moreover, \eqref{egalite} implies
$E_2=0$. In the sequel we denote by $C$ various positive constants
which may change from place to place but do not depend on $\ep$.
From Lemma \ref{properties-Y} $(ii)$---$(iv)$ we see that
$\left|\frac {(Y^m)_\xi}{Y_\xi}\right|=|mY^{m-1}|\leq C$ and that
$$\left|\frac {(Y^m)_{\xi\xi}}{Y_\xi}\right|\leq m(m-1)
|Y^{m-2}Y_\xi|+m Y^{m-1} \left|\frac
{Y_{\xi\xi}}{Y_\xi}\right|\leq C+C(e^{\mu(\ep M)t/\ep }-1)\,,$$
since $m\geq 2$. Hence
$$
E_3\geq (\CG \mu(\ep M)-C)\emuttplus -C\,.
$$
Since $\mu(\ep M)\to 1$ as $\ep \to 0$, by choosing $\CG \gg C$ we
see that $E_3\geq 0$ for all $\ep >0$ small enough.

Recalling that $Y_\xi >0$, we get ${\mathcal L}^\ep [w_\ep^+]\geq
0$ and the lemma is proved.
\end{proof}

We are now in the position to prove \eqref{facile} and
\eqref{sub-fisher}.

\begin{proof}
Let $\eta\in(0,1/2)$ be arbitrary. Then \cite[Lemma 3.11]{A-chemo}
provides a constant $C_Y>0$ such that, for all $\ep>0$ small
enough, for all $\xi\in (-C_0,C_0)$,
\begin{equation}\label{g-part11}
 Y( | \ln \ep |,\xi\,;\pm \ep M) \leq 1+\eta\,;
\end{equation}
\begin{equation}\label{g-part22}
\text{ if }\quad \xi\geq C_Y \ep\quad \text{ then }\quad Y(| \ln
\ep |,\xi\,;\pm \ep M)\geq 1-\eta\,.
\end{equation}

By setting $t= t^\ep=\ep|\ln \ep|$ in \eqref{g-coincee1}, we
obtain
\begin{eqnarray}
&&Y\left(|\ln \ep|, u_0(\Phi(0,t^\ep,x))-\CG \ep ^2 e^{\mu(-\ep M)|\ln \ep|} +\CG \ep ^2 \,;-\ep M\right) ^+\leq u^\ep(t^\ep,x)\nonumber\\
&&\leq Y\left(|\ln \ep|, u_0(\Phi(0,t^\ep,x))+\CG \ep ^2e^{\mu(\ep
M)|\ln \ep|} -\CG \ep ^2\,;\ep M\right)^+\,.\label{g-gr}
\end{eqnarray}
Therefore, the assertion \eqref{facile} of Theorem
\ref{th-gen-poreux} is a direct consequence of \eqref{g-gr} and
\eqref{g-part11}. Next we prove \eqref{sub-fisher}. Note that in
view of \eqref{mu}, we have $\ep e^{\mu(-\ep M)|\ln \ep|}\to 1$ as
$\ep \to 0$. Therefore, for $\ep >0$ small enough (since $Y_\xi
>0$),
\begin{equation}\label{blabla}
u^\ep(t^\ep, x)\geq Y\left(|\ln \ep|, u_0(\Phi(0,t^\ep,x)) -\frac
32 \CG \ep +\CG \ep ^2 \,;-\ep M\right) ^+\,.
\end{equation}
Choose $M_0\gg 0$ so that $M_0\ep- \frac 32\CG \ep+\CG \ep ^2 \geq
\max(C_Y \ep,a(-\ep M))$, with $C_Y$ as in \eqref{g-part22}. Then,
for any $x\in \om$ such that $u_0(\Phi(0,t^\ep,x))\geq M_0 \ep$,
we have
$$ u_0(\Phi(0,t^\ep,x))-\frac 32 \CG \ep +\CG \ep ^2\geq C_Y
\ep\,.$$ Combining this, \eqref{blabla} and \eqref{g-part22}, we
see that
\[
u^\ep(t^\ep, x)\geq 1-\eta\,.
\]
This completes the proof of \eqref{sub-fisher}.
\end{proof}

\subsection{Proof of \eqref{super-fish}}\label{ss:gen2}

Let us recall that a finite speed of propagation property, as is
\eqref{super-fish}, is proved in \cite{HKLM}: the authors
construct a super-solution using a related travelling wave $U$ of
minimal speed, and they obtain an $\mathcal O(\ep|\ln \ep|)$
estimate of the thickness of the transition layers. We borrow some
ideas from this paper but, in order to obtain the improved
$\mathcal O(\ep)$ estimate, we again use the solution $Y$ of the
ordinary differential equation \eqref{ode-poreux}.

\vskip 4pt Let $z^\ep$ be the solution of the Cauchy problem
(recall that $\ve(t,x)$ has been extended on $[0,\infty)\times \R
^N$ in Remark \ref{rem:extend})
\[
 \Qe \quad\begin{cases}
 z_t=\ep \Delta (z^m)-\n\cdot(z\n \ve)+\eun f(z)&\text{in } (0,\infty) \times\R ^N\vspace{3pt}\\
 z(0,x)=u_0(x) &\text{in }\R ^N\,.
 \end{cases}
\]

\begin{lem}[Super-solutions for $\Qe$ for small times]\label{finite-speed}
Choose $K\geq 1$ and $\CGBIS >0$ appropriately. For all $x_0 \in
\partial \Omega _0=
\partial {\rm Supp}\; u_0$, denote by $n_0$ the unit outward normal vector to
$\partial \Omega _0$ at $x_0$. For $t\geq 0$,  $x \in \R^n$,  define  the function
\begin{eqnarray*}
&&z_\ep ^+(t,x):=\nonumber\\
&&K \left[Y\left(\frac{t}{\ep},\,-(\Phi(0,t,x)-x_0)\cdot
n_0+\ep^2\CGBIS(\emuttplus-1)\,;\ep M\right)\right]^+\,.
\end{eqnarray*}
Here $Y(\tau,\xi\,;\delta)$ is the solution of \eqref{ode-poreux},
$\mm$ the slope defined in \eqref{slope}, $\Phi(t_1,t_2,x_3)$ the
flow defined in \eqref{flow} and $M$ is chosen such that, for
$\ep>0$ small enough, $M\geq C_0 \Vert \Delta \ve \Vert
_{L^\infty(\Q)}$. Then, for all $\ep
>0$ small enough,
\begin{equation}\label{initial}
u_0(x)\leq z_\ep ^+(0,x)\quad\text{ for all }\, x\in\R^N\,,
\end{equation}
and
\begin{equation}\label{proof-super}
{\mathcal L}^\ep [z_\ep ^+]:=(z_\ep ^+)_t-\ep \Delta ({z_\ep
^+})^m+\n \cdot(z_\ep ^+\n \ve)-\eun f(z_\ep ^+)\geq 0\,,
\end{equation}
in the domain $[0, t^\ep]\times \R ^N $.
\end{lem}

\begin{proof} Recall that $\Omega _0$ is convex. Therefore, in view of \eqref{pente}, we can choose $K\geq 1$
sufficiently large so that, for all $x_0 \in
\partial \Omega _0$ and all $x\in \Omega _0$,
\begin{equation}\label{convex}
u_0(x)\leq -K(x-x_0)\cdot n_0\,.
\end{equation}

We prove \eqref{initial}. If $\Phi(0,0,x)=x \notin \Omega _0$ this
is obvious since $u_0(x)=0$. Let us now assume that
$\Phi(0,0,x)=x\in \Omega _0$. Since $\Omega _0$ is convex, it lies
on one side of the tangent hyperplane at $x_0$ so that
$(x-x_0)\cdot n_0 <0$. Recall that $Y(0,\xi\,;\de)=\xi$ so that
$z_\ep ^+(0,x)=-K(x-x_0)\cdot n_0$ and \eqref{initial} follows
from \eqref{convex}.

We now prove \eqref{proof-super}. As in the proof of Lemma
\ref{g-w}, straightforward computations combined with
\eqref{ode-poreux} and \eqref{egalite} yield
\begin{eqnarray*}
&&\ep {\mathcal L}^\ep [z_\ep^+]=Kf(Y)-f(KY)+\ep K(M+Y\Delta
\ve) \\
&&+ \ep ^2 K Y_\xi   \Big\{\CGBIS \mu(\ep M) \, \emuttplus
-D_3\Phi(0,t,x)n_0\cdot \n v_1^\ep(t,x)\\
&&+\Delta _x[\Phi(0,t,x)\cdot n_0]K^{m-1}
\frac{(Y^m)_\xi}{Y_\xi}-|\n _x[\Phi(0,t,x)\cdot n_0]|^2
K^{m-1}\frac{(Y^m)_{\xi\xi}}{Y_\xi}\Big\}\,.
\end{eqnarray*}
Note that $Kf(Y)-f(KY)=K(K-1)Y^2\geq 0$. Then, by using similar
arguments to those in the proof of Lemma \ref{g-w}, we see that
${\mathcal L}^\ep [z_\ep^+]\geq 0 $, if $\CGBIS>0$ is sufficiently
large.
\end{proof}

We now prove \eqref{super-fish}.

\begin{proof}
We shall first prove that property \eqref{super-fish} holds for
$z^\ep$ the solution of the Cauchy Problem $\Qe$. Recall that
$a(\de)$ is the unstable zero of $\tilde f _\de=\tilde f+\de$ so
that $a(\ep M)<0$. Moreover, in view of \eqref{h} and \eqref{mu},
we can choose $M_0>0$ large enough so that, for $\ep >0$ small
enough,
$$
-M_0 \ep+\CGBIS \ep e^{(\mu(\ep M)-1)|\ln \ep|}-\CGBIS \ep
^2<a(\ep M)\,.
$$
For $x \in\Omega$ such
that ${\rm dist}(\Phi(0,t^\ep,x),\Omega _0)\geq M_0 \ep$, we
choose $x_0 \in
\partial \Omega _0$ such that ${\rm dist}(\Phi(0,t^\ep,x),\Omega _0)=\Vert \Phi(0,t^\ep,x)-x_0\Vert$ and define $z_\ep ^+$ as in
Lemma \ref{finite-speed}. It follows from Lemma \ref{finite-speed}
and the comparison principle that, for all $\ep
>0$ small enough, all $(t,x) \in[0, t^\ep]\times \R^N$,
\begin{equation}\label{coincee-gene}
0\leq z^\ep (t,x)\leq z_\ep ^+(t,x)\,.
\end{equation}
Since, for $t=t^\ep= \ep  | \ln \ep |$,
\begin{eqnarray*}
&-&(\Phi(0,t^\ep,x)-x_0)\cdot
n_0+\ep^2\CGBIS(\emuttplusep-1)\\
&&=-\Vert \Phi(0,t^\ep,x)-x_0\Vert+\CGBIS \ep e^{(\mu(\ep
M)-1)|\ln \ep|}-\CGBIS
\ep ^2\\
&&\leq -M_0 \ep+\CGBIS \ep e^{(\mu(\ep M)-1)|\ln \ep|}-\CGBIS
\ep ^2\\
&&<a(\ep M)\,,
\end{eqnarray*}
it follows from Lemma \ref{properties-Y} $(i)$ that $$
Y\left(\frac{t^\ep}{\ep},\,-(\Phi(0,t^\ep,x)-x_0)\cdot
n_0+\ep^2\CGBIS(\emuttplusep-1)\,;\ep M\right)<a(\ep M)<0\,,$$ and
therefore $z_\ep^+(t^\ep,x)=0$, which in turn implies
$z^\ep(t^\ep,x)=0$. Hence \eqref{super-fish} holds for $z^\ep$ the
solution of $\Qe$.

Now, a straightforward modification of \cite[Corollary 4.1]{HKLM}
shows that there exists $\tilde T>0$ such that, for all $\ep >0$
small enough,
$$
\ue(t,x)=z^\ep(t,x)\,,
$$
for all $(t,x)\in(0,\tilde T)\times \Omega$. This proves
\eqref{super-fish} for $\ue$ the solution of $\Pe$.
\end{proof}

\section{The propagating front}\label{s:motion}

The goal of this section is to construct efficient sub- and
super-solutions that control $\ue$ during the latter time range,
when the motion of interface occurs. We begin with some
preparations.

\subsection{Materials}\label{ss:materials}

In the linear diffusion case ($m=1$), it is well-known that the
equation $u_t=\Delta u +u(1-u)$ admits travelling wave solutions
with some semi-infinite interval of admissible wave speed. The
same property holds for the nonlinear diffusion case, namely
equation $u_t=\Delta (u^m)+u(1-u)$, $m>1$. Nevertheless, it turns
out that the travelling wave with minimal speed $c^*>0$ is both
compactly supported from one side and sharp. In the following, $U$
denotes the unique solution of
\begin{equation}\label{tw}
\begin{cases}
(U^m)''(z)+c^* U'(z)+U(z)(1-U(z))=0\quad \text{ for all } z\in \R\\
 U(-\infty)=1\\
  U(z)>0 \quad\text{ for all } z<0\\
 U(z)=0\quad \text{ for all } z\geq 0\,.
 \end{cases}
\end{equation}

\begin{lem}[Behavior of $U$]\label{behavior-tw}For all
$z\in(-\infty,0)$ we have $U'(z)<0$. The travelling wave $U$ is
smooth outside 0 and
$$
U'(0)\quad\begin{cases}
 =0 &\text{ if }\;1<m<2\vspace{3pt}\\
 \in(-\infty,0) &\text{ if }\;m=2\vspace{3pt}\\
 =-\infty &\text{ if }\;m>2\,.
 \end{cases}
 $$
Moreover, there exist $C>0$ and $\beta >0$ such that the following
properties hold.
\begin{equation}\label{prop1}|(U^m)'(z)|\leq C U(z)\,\quad\text{
for all }\, z\in \R\,,
\end{equation}
\begin{equation}\label{prop2}0< 1-U(z)\leq Ce^{-\beta |z|}\,\quad\text{
for all }\, z \leq 0\,,
\end{equation}
\begin{equation}\label{prop3} |zU'(z)|\leq CU(z)\,\quad\text{
for all }\, z \leq -1\,.
\end{equation}
\end{lem}
For more details and proofs we refer the reader to
\cite{Atk-Reu-Rid}, \cite{Bir}, \cite{HKLM}, as well as to
\cite{mama} for related results.

\vskip 4pt  Another ingredient is a \lq\lq cut-off signed distance
function" $d^\ep (t,x)$ which is defined as follows. Let
$\widetilde d^\ep=\widetilde d\, ^{\ep,\text{drift}}$ be the
signed distance function to $\Gammadrift$, the smooth solution of
the free boundary problem $\drift$, namely
\begin{equation}\label{eq:dist}
\widetilde d ^\ep (t,x):=
\begin{cases}
-&\hspace{-10pt}{\rm dist}(x,\Gammadrift_t)\quad\text{ for }x\in\Omegadrift  _t \\
&\hspace{-10pt} {\rm dist}(x,\Gammadrift _t) \quad \text{ for }
x\in\Omega \setminus \overline{\Omegadrift  _t}\,,
\end{cases}
\end{equation}
where ${\rm dist}(x,\Gammadrift _t)$ is the distance from $x$ to
the hypersurface  $\Gammadrift  _t$. We remark that $\widetilde
d^\ep=0$ on $\Gammadrift $ and that $|\nabla \widetilde d^\ep|=1$
in a neighborhood of $\Gammadrift$: there exists $d_0
>0$ such that, for all $\ep >0$ small enough, $|\n \widetilde
d^\ep (t,x)|=1$ if $|\widetilde d^\ep (t,x)|<2d_0$. By reducing
$d_0$ if necessary we can assume that $\widetilde d^\ep$ is smooth
in $ \{(t,x) \in
 [0,T]\times \overline \Omega :\,|\widetilde d^\ep (t,x)|<3
 d_0\}\,$.

 Next, let $\zeta(s)$ be a smooth increasing function on $\R$ such
that
\[
 \zeta(s)= \left\{\begin{array}{ll}
 s &\textrm{ if }\ |s| \leq d_0\vspace{4pt}\\
 -2d_0 &\textrm{ if } \ s \leq -2d_0\vspace{4pt}\\
 2d_0 &\textrm{ if } \ s \geq 2d_0\,.
 \end{array}\right.
\]
We then define the cut-off signed distance function $d^\ep= d\,
^{\ep,\text{drift}}$ by
\begin{equation}
d^\ep(t,x):=\zeta\left(\widetilde d^\ep(t,x)\right)\,.
\end{equation}

Note that
\begin{equation}\label{norme-un}
\text{ if } \quad |d^\ep(t,x)|< d_0 \quad \text{ then }\quad
|\nabla d^\ep(t,x)|=1\,,
\end{equation}
and that the equation of motion $\drift$ is recast as
\begin{equation}\label{interface}
  \left(d^\ep _t+c^*+\n d^\ep\cdot \n v\right) (t,x)=0 \quad\
 \textrm{ on}\ \; \Gammadrift _t=\{x \in \Omega :\,
d^\ep(t,x)=0\}\,.
\end{equation}
Moreover, there exists a constant $D>0$ such that, for all $\ep
>0$ small enough,
\begin{equation}\label{est-dist}
|\nabla d^\ep (t,x)|+|\Delta d^\ep (t,x)|\leq D\quad \textrm{ for
all } (t,x) \in \overline {\Q}\,.
\end{equation}
Finally, in view of \eqref{norme-un} and \eqref{interface}, the
mean value theorem provides a constant $N>0$ such that, for all
$\ep
>0$ small enough,
\begin{equation}\label{MVT}
\left| d^\ep_t+c^*|\n d^\ep |^2+\n d^\ep\cdot \n v\right|(t,x)\leq
N|d^\ep(t,x)| \quad \textrm{ for all } (t,x) \in \overline {\Q}\,.
\end{equation}

\subsection{Sub- and super-solutions}

We  define
\begin{equation}\label{sup-sol}
u_\ep ^ \pm(t,x):=(1\pm q(t) )U\left(\frac{d^\ep(t,x)\mp \ep
p(t)}\ep\right)\,,
\end{equation}
where
$$
\begin{array}{ll}
p(t):=-e^{-t/\ep}+e^{Lt}+K\,\vsp\\
q(t):=\sigma(e^{-t/\ep}+\ep L e^{Lt})\,,
\end{array}
$$
and where $U$ and $d^\ep$ were defined in subsection
\ref{ss:materials}. In the following lemma, $\support[u_\ep
^\pm]$, $\gammasupport[u_\ep ^\pm]$, $\Gamma ^{supp}[u_\ep ^\pm]$
and $\nusupport[u_\ep ^\pm]$ are defined as in Lemma
\ref{lemma-sup-poreux}.

\begin{lem}[Sub- and super-solutions for the propagating front]\label{super-sub-motion1}
Choose $\sigma >0$ small enough so that
\begin{equation}\label{sigma}
c^*(m-1)D^2(1+2\sigma)^{m-2}\sigma \leq \frac 1 2\,,
\end{equation}
where $D$ is the constant that appears in \eqref{est-dist}. Choose
$K\geq 1$. Then, if $L>0$ is large enough, we have, for $\ep >0$
small enough,
\begin{equation}\label{sup}
 \mathcal L ^\ep  [u_\ep ^-]\leq 0 \leq  \mathcal L ^\ep  [u_\ep ^+]\, \quad
\text{ in }\,  (0,T) \times\Omega\,,
\end{equation}
\begin{equation}\label{bords-sub}
\frac{\partial (u_\ep ^-)^m}{\partial \nusupport[u_\ep ^-]}=0\;
\text{ on }\
\partial \support[u_\ep ^-]\,\; \text{ for all }\ t\in[0,T]\,,
\end{equation}
\begin{equation}\label{bords-sup}
\frac{\partial (u _\ep ^+)^m}{\partial \nusupport[u_\ep ^+]}=0\;
\text{ on }\
\partial \support[u _ \ep ^+]\,\; \text{ for all }\ t\in[0,T]\,.
\end{equation}
\end{lem}

\begin{proof} Properties \eqref{bords-sub} and \eqref{bords-sup} follow from
$(U^m)'(0)=0$ (see \eqref{prop1}). We prove below that $\mathcal L
^\ep [u_\ep ^+]\geq 0$, the proof of $\mathcal L ^\ep  [u_\ep
^-]\leq 0$ following the same lines.  Note that we only need to
consider the case $d^\ep(t,x)\leq \ep p(t)$ since, if
$d^\ep(t,x)>\ep p(t)$ then $u _ \ep ^+(t,x)=0$. Straightforward
computations and equation \eqref{tw} yield
$$
\begin{array}{lll}
\ep(u_\ep^+)_t=\ep q'U+(1+q)(d^\ep_t-\ep p')U'\vsp \\
\ep \n u_\ep ^+=(1+q) U'\n d^\ep  \vsp\\
 \ep ^2 \Delta
(u_\ep^+)^m=(1+q)^m|\n d^\ep| ^2(-c^*U'-U(1-U))+\ep(1+q)^m\Delta
d^\ep (U^m)' \,,
\end{array}
$$
where arguments are omitted. Thus we get
$$
\ep \mathcal L ^\ep [u _ \ep ^+]=E_1+E_2+E_3\,,
$$ where
\vsp\\
$ E_1=U'(1+q)\left[d^\ep _t - \ep p'+c^*(1+q)^{m-1}|\n d^\ep|^2+\n d^\ep\cdot \n \ve\right]=:U'(1+q)E_1^\star$\vsp \\
$ E_2=U\left\{-(1+q)+(1+q) ^m|\n d^\ep|^2+U\left[(1+q)^2-(1+q)^m|\n d^\ep|^2\right]+\ep q'\right\}$\vsp \\
$ E_3=-\ep (1+q)^m \Delta d^\ep (U^m)'+\ep (1+q) \Delta \ve U$\,.\vsp\\
In the sequel we define $a(t):=1+q(t)$ and denote by $C_i$ various
positive constants which do not depend on $\ep$.

\vskip 4pt Since $\Vert \Delta \ve \Vert _{L ^\infty(\Q)}$ is
uniformly bounded w.r.t. $\ep >0$ (see Assumption \ref{H3}), we
deduce from \eqref{est-dist} and \eqref{prop1} that $|E_3|\leq \ep
C_3 (a^m+a) U$ so that
\begin{multline}\label{mult}
E_2+E_3\geq U\Big\{-a+a^m+U\left(a^2-a^m\right)-\ep C_3
a^m-\ep C_3 a\\
+\left(|\n d^\ep| ^2-1\right)a^m(1-U)+\ep q'\Big\}\,.
\end{multline}
We claim that, for $\ep >0$ small enough,
\begin{equation}\label{truc}
|(|\n d^\ep|^2-1)(1-U)|\leq \ep C_2\,.
\end{equation}
Indeed, if $-d_0<d^\ep(t,x)\leq \ep p(t)$, it follows from
\eqref{norme-un} that, for $\ep >0$ small enough, $|\n d^\ep
(t,x)|=1$. Next, if $d^\ep(t,x)\leq -d_0$, \eqref{prop2} implies
that
$$
0<(1-U)\left(\frac{d^\ep(t,x)-\ep p(t)}\ep\right)\leq
(1-U)(-\frac{d_0}\ep) \leq C e^{-\beta\frac{d_0}\ep}\,,
$$
and \eqref{truc} holds for $\ep>0$ small enough. Therefore
\eqref{mult} and \eqref{truc} imply
\begin{equation}
E_2+E_3\geq U\left\{-a+a^m+U\left(a^2-a^m\right)-\ep (C_2+ C_3)
a^m-\ep C_3 a+\ep q'\right\}\,.
\end{equation}

\vskip 4pt Next, since
$$
E_1^\star=d^\ep_t+c^*|\n d^\ep| ^2 +\n d^\ep\cdot \n v-\ep
p'+c^*(a^{m-1}-1)|\n d^\ep| ^2+\ep \n d^\ep \cdot \n v_1 ^\ep\,,
$$
using \eqref{MVT}, \eqref{est-dist} and Assumption \ref{H3}, we
see that \begin{eqnarray*}
E_1^\star &\leq& N|d^\ep(t,x)|-\ep p'(t)+c^*(a^{m-1}-1)|\n d^\ep|^2+\ep \n d^\ep \cdot \n v_1 ^ \ep\\
&\leq& N|d^\ep(t,x)-\ep p(t)|+\ep
(Np(t)-p'(t))+c^*(a^{m-1}-1)D^2+\ep CD\\
&\leq& N|d^\ep(t,x)-\ep p(t)|+\ep
(Np(t)-p'(t))\\
&{}& +c^*(m-1)D^2(1+2\sigma)^{m-2} q(t)+\ep  CD\,.
\end{eqnarray*}
The last inequality above comes from the fact that, for $\ep>0$
small enough, we have $0\leq q(t)\leq \sigma (1+\ep L e^{LT})\leq
2\sigma$, which in turn implies that
\begin{equation}\label{q-borne}
0\leq a^{m-1}-1\leq (m-1)(1+2\sigma)^{m-2} q(t)\,.
\end{equation}
In the following, we distinguish two cases.

First, assume that $0\leq d^\ep(t,x) \leq \ep p(t)$ so that, for
$\ep
>0$ small enough,
\begin{eqnarray*}
E_1^\star&&\leq \ep (2Np(t)-p'(t))+c^*(m-1)D^2(1+2\sigma)^{m-2} q(t)+\ep  CD\\
&&\leq e^{-t/\ep}(-\ep
2N-1+c^*(m-1)D^2(1+2\sigma)^{m-2}\sigma)\\
&&\quad+e^{Lt}(\ep 2N -\ep L+\ep
c^*(m-1)D^2(1+2\sigma)^{m-2}\sigma L)+\ep 2N K+\ep CD\,.
\end{eqnarray*}
In view of \eqref{sigma} we get
$$
E_1^\star \leq \ep\left(e^{Lt}(2N-\frac 12 L)+2NK+CD\right)\leq
0\,,
$$
if $L>0$ is sufficiently large. This implies that
$E_1=aU'E_1^\star\geq 0$.

Now, assume that $d^\ep(t,x)\leq 0$ so that
\begin{equation}\label{moinsun}
\frac {d^\ep(t,x)-\ep p(t)}\ep \leq -K \leq -1\,.
\end{equation}If
$E_1^\star \leq 0$ the conclusion $E_1\geq 0$ follows. Let us now
assume $E_1^\star > 0$. The above study for the case $0\leq
d^\ep(t,x)\leq \ep p(t)$ implies a fortiori that
$$\ep(Np(t)-p'(t))+c^*(m-1)D^2(1+2\sigma)^{m-2} q(t)+\ep  CD\leq 0\,.
$$
Therefore
$$
|E_1^\star|\leq N|d^\ep(t,x)-\ep p(t)|\,,
$$
and we deduce from \eqref{moinsun} and \eqref{prop3} that
$$
|E_1|\leq \ep C_1 a U\,.
$$

\vskip 4pt Summarizing, we obtain that, in any cases,
$$
\ep \mathcal L ^\ep [u _\ep ^+]\geq
 U\left\{-a+a^m+U\left(a^2-a^m\right)-\ep C_ 4  a^m+\ep
q'\right\}\,,
$$
since $a=1+q> 1$ and with $C_4:=C_1+C_2+2C_3$. Since $U<1$, $a> 1$
and $m\geq 2$, the inequality $-a+a^m+U\left(a^2-a^m\right) \geq
-a+a^2=q+q^2\geq q$ holds. Therefore, using $|a|\leq 1+2\sigma$
and substituting the expressions for $q(t)$ and $q'(t)$, we see
that
$$
\begin{array}{ll}
\ep \mathcal L ^\ep [u _ \ep ^+]&\geq
U\left\{\ep \sigma L e^{Lt}-\ep C_4 (1+2\sigma)^m+\sigma \ep ^2 L ^2e^{Lt}\right\}\vsp\\
&\geq U\ep \left\{\sigma L-C_4 (1+2\sigma)^m\right\}\vsp\\
&\geq 0\,,
\end{array}
$$
if $L>0$ is sufficiently large.

This completes the proof of Lemma \ref{super-sub-motion1}.
\end{proof}

\section{Proof of Theorem \ref{th:results}}\label{s:proof}

By fitting the sub- and super-solutions for the generation into
the ones for the motion, we are now in the position to prove our
main result.

\vskip 4pt Let $\eta \in (0,1/2)$ be arbitrary. Choose $\sigma$
that satisfies \eqref{sigma} and
\begin{equation}\label{eta}
\sigma \leq \frac \eta 2\,.
\end{equation}
By Theorem \ref{th-gen-poreux}, there exists $M_0>0$ such that
\eqref{facile}, \eqref{sub-fisher} and \eqref{super-fish} hold
with the constant $\eta$ replaced by $\sigma /2$. Recall that
$u_\ep ^\pm$ are the sub- and super-solutions constructed in
\eqref{sup-sol}.

\begin{lem}[Ordering initial data]\label{condition-initiale2}
There exists $\tilde K >0$ such that for all $K\geq \tilde K$, all
$L>0$, all $\ep>0$ small enough, we have
\begin{equation}
u_\ep ^-(0,x)\leq\ue(t^\ep, x)\leq u_\ep ^+(0,x)\,,
\end{equation}
for all $x\in  \Omega$.
\end{lem}

\begin{proof}
We first prove
\begin{equation}\label{goal}
u_\ep ^-(0,x)=(1-\sigma -\ep \sigma L)U\left(\frac{d^\ep(0,x)+
K\ep }\ep\right)\leq \ue(t^\ep, x)\,.
\end{equation}
If $x$ is such that $d^\ep(0,x)\geq -K\ep$, this is obvious since
the definition of $U$ then implies $u_\ep ^-(0,x)=0$. Next assume
that $x$ is such that $d^\ep(0,x)< -K \ep$. Let us denote by
$d(t,x)$ the signed distance function to $\Gamma _t$. Note that,
in view of hypothesis \eqref{pente}, the mean value theorem
provides the existence of a constant $\tilde K_0>0$ such that
\begin{equation}\label{ineg}
\text{ if } \quad d(0,y)\leq - \tilde K_0 \ep \quad \text{ then }
\quad u_0(y)\geq M_0 \ep\,.
\end{equation}
Moreover in view of the definition of $\intdrift$ in
\eqref{drifted-interface} and the compactness of $\Gamma _0$,
there exists $K_0
>0$ such that, for $\ep >0$ small enough,
\begin{equation}\label{ineg2}
\text{ if } \quad d^\ep(0,x)\leq -  K_0 \ep \quad \text{ then }
\quad d(0,\Phi(0,t^\ep,x)) \leq - \tilde K _0 \ep\,.
\end{equation}
Hence, if we choose $K \geq K_0$, we deduce from \eqref{ineg2},
\eqref{ineg} and \eqref{sub-fisher} (with $\eta$ replaced by
$\sigma /2$) that $\ue(t^\ep,x)\geq 1 - \frac \sigma 2$. Since
$U\leq 1$, this proves \eqref{goal}.

Next we prove
\begin{equation}\label{goal2}
\ue(t^\ep, x)\leq (1+\sigma+\ep\sigma L)U\left(\frac{d^\ep(0,x)-
K\ep }\ep\right)=u_\ep ^+(0,x)\,.
\end{equation}
In view of the definition of $\intdrift$ in
\eqref{drifted-interface} and the compactness of $\Gamma _0$,
there exists $K_1 >0$ such that
\begin{equation}\label{ineg3}
\text{ if } \quad d^\ep(0,x)\geq  K_1 \ep \quad \text{ then }
\quad {\rm dist}(\Phi(0,t^\ep,x),\Omega _0) \geq M _0 \ep\,.
\end{equation}
If $x$ is such that $d^\ep(0,x)\geq K_1 \ep$ then it follows from
\eqref{ineg3} and Theorem \ref{th-gen-poreux} $(iii)$ that
$\ue(t^\ep, x)=0$, which proves \eqref{goal2}. Next assume that
$x$ is such that $d^\ep(0,x)< K_1 \ep$. Since $U$ is non
increasing we have
$$
(1+\sigma+\ep \sigma L)U\left(\frac{d^\ep(0,x)- K \ep
}\ep\right)\geq (1+\sigma)U(K_1-K) \geq 1+\frac \sigma 2\,,
$$
if $K\gg K_1$ (recall that $U(-\infty)=1$). Then \eqref{goal2}
follows from \eqref{facile} (with $\eta$ replaced by $\sigma /2$).
\end{proof}

We now prove Theorem \ref{th:results}.

\begin{proof}We fix $K\geq 1$ large enough so that Lemma
\ref{condition-initiale2} holds, and $L>0$ large enough so that
Lemma \ref{super-sub-motion1} holds. Therefore, from the
comparison principle, we deduce that
\begin{equation}\label{ok}
u_\ep ^-(t,x) \leq \ue (t+t^\ep,x) \leq u _ \ep ^+(t,x) \quad
\text { for }\; 0 \leq t \leq T-t^\ep\,.
\end{equation}
Note that, since
\begin{equation}\label{sub-lim}
\lim_{\ep\rightarrow 0} u _ \ep ^\pm(t,x)= \left\{
\begin{array}{ll}
1 &\textrm { if } d^\ep(t,x)<0\vsp\\
0 &\textrm { if } d^\ep(t,x)>0\,,\\
\end{array}\right.
\end{equation}
for $t>0$, \eqref{ok} is enough to prove the convergence result,
namely Corollary \ref{cor:cv}. We now choose $\mathcal C$ large
enough so that
\begin{equation}\label{choice}
(1-\frac 3 4 \eta )U(-\mathcal C +e^{LT} +K)\geq 1- \eta\,.
\end{equation}
Note that this choice forces
\begin{equation}\label{choice-bis}
\mathcal C \geq e^{LT}+K\,.
\end{equation}
In the sequel we prove \eqref{resultat}.

Obviously, if $\ep >0$ is small enough, the constant map
$z^+\equiv 1+\eta$ is a super-solution. Therefore we deduce from
Theorem \ref{th-gen-poreux} $(i)$ that $\ue(t+t^\ep,x) \in
[0,1+\eta]$, for all $0\leq t \leq T-t^\ep$.

Next we take $x\in\Omegadrift  _ t\setminus\mathcal N_{\mathcal
C\ep}(\Gammadrift _t)$, i.e.
\begin{equation}\label{d-plus}
d^\ep(t,x)\leq -\mathcal C \ep\,.
\end{equation}
For $\ep >0$ small enough, we have
$$
\begin{array}{ll}
 u_\ep ^-(t,x)&\geq (1-\sigma-\ep \sigma L e^{LT})U(-\mathcal
C +e^{LT}+K)\vsp\\
&\geq (1-\frac 32 \sigma)U(-\mathcal
C +e^{LT}+K)\vsp\\
&\geq (1-\frac 34\eta)U(-\mathcal C +e^{LT}+K)\vsp\\
&\geq 1-\eta\,,
\end{array}
$$
where we have successively used \eqref{eta} and \eqref{choice}. In
view of \eqref{ok} this implies that $\ue(t+t^\ep,x)\geq 1-\eta$,
for all $0\leq t \leq T-t^\ep$.

Finally we take $x\in (\Omega \setminus \overline{\Omegadrift
_t})\setminus\mathcal N_{\mathcal C\ep}(\Gammadrift _t)$, i.e.
\begin{equation}\label{d-moins}
d^\ep(t,x)\geq \mathcal C \ep\,.
\end{equation}
Using \eqref{choice-bis} we see that, for $\ep >0$ small enough,
$d^\ep(t,x)-\ep p(t) \geq 0$ so that  $u _ \ep ^+(t,x)=0$, which,
in view of \eqref{ok} implies that $\ue(t+t^\ep,x)=0$, for $0\leq
t \leq T-t^\ep$.

This completes the proof of Theorem \ref{th:results}. \end{proof}

\vskip 4pt \noindent {\bf Acknowledgement.} The first author is
grateful to  Hiroshi Matano for a fruitful  discussion in Dresden.

\end{document}